\documentclass[a4paper,11pt,reqno]{amsart}

\usepackage{amssymb, tikz}
\usetikzlibrary{decorations.markings}

\usepackage{hyperref}

\numberwithin{equation}{section}

\newtheorem{theorem}{Theorem}[section]
\newtheorem{corollary}[theorem]{Corollary}
\newtheorem{lemma}[theorem]{Lemma}
\newtheorem{proposition}[theorem]{Proposition}
\theoremstyle{remark}
\newtheorem{remark}[theorem]{Remark}
\newtheorem{example}[theorem]{Example}
\theoremstyle{definition}

\newcommand\bp{\begin{proof}}
\newcommand\ep{\end{proof}}

\newcommand\unb{{\mathbf 1}}

\newcommand\Aut{\operatorname{Aut}}
\newcommand\Mat{\operatorname{Mat}}
\newcommand\Per{\operatorname{Per}}

\newcommand\mueq{\mu_{\mathrm{eq}}}
\newcommand\nueq{\nu_{\mathrm{eq}}}
\newcommand\meae{\mu_{\mathrm{eq}}\text{-}\mathrm{a.e.{}}}

\newcommand{\C}{{\mathbb C}}
\newcommand{\N}{{\mathbb N}}
\newcommand{\R}{{\mathbb R}}
\newcommand\T{{\mathbb T}}
\newcommand\Z{{\mathbb Z}}

\newcommand{\Ff}{{\mathcal F}}
\newcommand\G{\mathcal G}
\newcommand\RR{\mathcal R}
\newcommand\QQ{\mathcal Q}

\def\Bb{\mathcal{B}}

\def\Ff{\mathcal{F}}

\def\Oo{\mathcal{O}}
\def\Pp{\mathcal{P}}
\def\Tt{\mathcal{T}}

\newcommand{\Hh}{\mathcal{H}}
\newcommand{\Kk}{\mathcal{K}}

\newcommand{\clsp}{\overline{\mathrm{span}}}

\newcommand{\Lmin}{\Lambda^{\min}}

\def\lms{\Lambda^0\!/_{\!\!\sim}}
\def\lmss{\Lambda^0\!/_{\!\!\approx}}

\begin{document}

\title{Von Neumann algebras of strongly connected higher-rank graphs}

\date{September 22, 2014; minor corrections on January 13, 2015}

\author[Laca]{Marcelo Laca}
\address[Marcelo Laca]{
Department of Mathematics and Statistics\\
University of Victoria\\
PO Box 3060, Victoria BC,
Canada V8W 3R4}%
\email{laca@math.uvic.ca}

\author[Larsen]{Nadia S. Larsen}
\address[Nadia Larsen]{Department of Mathematics\\
University of Oslo\\
PO BOX 1053 Blindern\\
N-0316 Oslo\\
Norway}%
\email{nadiasl@math.uio.no}

\author[Neshveyev]{Sergey Neshveyev}
\address[Sergey Neshveyev]{Department of Mathematics\\
University of Oslo\\
PO BOX 1053 Blindern\\
N-0316 Oslo\\
Norway}%
\email{sergeyn@math.uio.no}

\author[Sims]{Aidan Sims}
\address[Aidan Sims]{
School of Mathematics and Applied Statistics\\
University of Wollongong\\
NSW 2522\\
Australia}%
\email{asims@uow.edu.au}

\author[Webster]{Samuel B.G. Webster}
\address[Sam Webster]{
School of Mathematics and Applied Statistics\\
University of Wollongong\\
NSW 2522\\
Australia}%
\email{sbgwebster@gmail.com}

\thanks{This research was supported by the Australian Research Council and the Natural Sciences
and Engineering Research Council of Canada. Parts of the work were completed at the
workshop \emph{Operator algebras and dynamical systems from number theory} (13w5152) at
the Banff International Research Station in November 2013 and at the conference
\emph{Classification, Structure, Amenability and Regularity} at the University of Glasgow
in September 2014, supported by the EPSRC, GMJT and LMS}

\begin{abstract}
We investigate the factor types of the extremal KMS states for the preferred dynamics on
the Toeplitz algebra and the Cuntz--Krieger algebra of a strongly connected finite
$k$-graph. For inverse temperatures above~1, all of the extremal KMS states are of
type~$\mathrm{I}_\infty$. At inverse temperature~1, there is a dichotomy: if the
$k$-graph is a simple $k$-dimensional cycle, we obtain a finite type~$\mathrm{I}$ factor;
otherwise we obtain a type~III factor, whose Connes invariant we compute in terms of the
spectral radii of the coordinate matrices and the degrees of cycles in the graph.
\end{abstract}

\maketitle

\section{Introduction}\label{sec:intro}

The $C^*$-algebras of strongly connected finite higher-rank graphs provide interesting
higher-rank analogues of Cuntz--Krieger algebras. In this paper we study von Neumann
algebras generated by these $C^*$-algebras in the representations defined by the extremal
KMS states for the preferred dynamics studied in \cite{AHLRS2}. Results of Enomoto, Fujii
and Watatani~\cite{EFW} show that when $k = 1$ and the graph is not a simple cycle, there
is a unique KMS state and the associated factor is of type~$\mathrm{III}_{\rho(A)^{-p}}$,
where $\rho(A)$ is the spectral radius of the adjacency matrix $A$ of the graph, and $p$
is the period of the graph in the sense of Perron-Frobenius theory: the greatest common
divisor of the lengths of cycles in the graph.

In the higher-rank case there can be more than one KMS state, and a complete
classification of such states has been recently obtained in~\cite{AHLRS2}. Specifically,
Theorem~7.1 of \cite{AHLRS2} shows that the extremal KMS states of the $C^*$-algebra
$C^*(\Lambda)$ of a finite strongly connected $k$-graph $\Lambda$ are indexed by the
characters of an associated subgroup $\Per\Lambda$ of $\Z^k$, whose group $C^*$-algebra
embeds as a central subalgebra of $C^*(\Lambda)$. The goal of the present paper is to
determine the types of these KMS states. Using Feldman--Moore theory~\cite{FM2} and the
groupoid description of a $k$-graph algebra, we obtain a very satisfactory generalisation
of Enomoto, Fujii and Watatani's result. Namely, suppose that~$\Lambda$ is not a simple
$k$-dimensional cycle. We define $\Pp_\Lambda$ to be the subgroup of $\Z^k$ generated by
the degrees of cycles in $\Lambda$. The $k$ coordinate graphs of $\Lambda$ determine
integer matrices $A_i$. The vector $\rho(\Lambda) = (\rho(A_1), \dots, \rho(A_k))$ of
spectral radii of these matrices determines a homomorphism $n \mapsto \rho(\Lambda)^n$ of
$\Pp_\Lambda$ into the multiplicative group of positive reals. We prove that the closure
of its image is the Connes spectrum of the type~III factor obtained from any of the
extremal KMS states of $C^*(\Lambda)$ described in \cite{AHLRS2}. We also determine the
types of the factors arising from KMS states on the Toeplitz algebra that do not factor
through $C^*(\Lambda)$, and from KMS states of $C^*(\Lambda)$ when $\Lambda$ is a simple
$k$-dimensional cycle. An interesting corollary is that the factors obtained from a
$k$-graph~$\Lambda$ depend only on its skeleton, and are independent of the factorisation
property.

In the case when $\Lambda$ is primitive and aperiodic --- or equivalently, when
$\Pp_\Lambda=\Z^k$ and $\Per\Lambda=0$ --- the unique KMS state $\varphi$ is the most
natural state on $C^*(\Lambda)$: it is the unique gauge-invariant state whose restriction
to the AF core of $C^*(\Lambda)$ is tracial. By our result, the Connes spectrum of the
associated factor is then the closure of the multiplicative group generated by the
spectral radii~$\rho(A_i)$ of the connectivity matrices $A_i$. In some special cases this
has been already established by Yang~\cite{Y0,Y}. She studied $C^*$-algebras and von
Neumann algebras of aperiodic $k$-graphs with a single vertex --- the higher-rank
analogues of Cuntz algebras. Under the technical condition that the so-called ``intrinsic
group" of the graph has rank at most 1, she proved that $\varphi$ is a factor state of
type~III with Connes spectrum equal to the closure of the multiplicative group generated
by the numbers $m_1, \dots, m_k$ of edges of each of the $k$ minimal degrees. This
generalises Olesen and Pedersen's result~\cite{OP} that the unique KMS state for the
gauge-action on the Cuntz algebra $\Oo_n$ is a type~$\mathrm{III}_{1/n}$ factor state.
Yang's result completely resolved the situation for aperiodic single-vertex $2$-graphs.
She then asked whether the result remains true for all single-vertex $k$-graphs,
regardless of the intrinsic group. A special case of our main theorem implies that this
is indeed the case, under the sole assumption of aperiodicity.
\newline

The paper is organised as follows. We introduce necessary background about $k$-graphs and
their $C^*$-algebras in Section~\ref{sprelim}. In Section~\ref{sec:mainthm}, we state our
main result, Theorem~\ref{thm:main}, and begin the proof by analysing the factors arising
from KMS states at large inverse temperatures. These are all type~$\mathrm{I}_\infty$
states and the associated von Neumann factors each have a canonical presentation as
$\Bb(\ell^2(\Lambda v))$ for some $v \in \Lambda^0$.

In Section~\ref{sec:S-general} we present a computation of the Connes invariant
$S(W^*(\QQ))$ of the von Neumann algebra of an ergodic countable equivalence relation
$\QQ$ with a quasi-invariant measure~$\mu$. Corollary~\ref{ceqConnes} says that if the
sub-relation $\QQ^D$ defined by the kernel of the Radon--Nikodym cocycle of $\mu$ is
ergodic, then $S(W^*(\QQ))$ is precisely the essential range of the Radon--Nikodym
cocycle. These results are surely known, but we give a self-contained treatment in lieu
of an explicit reference.

In Section~\ref{sequiv} we apply groupoid methods to study the factors associated to the
extremal KMS states of~$C^*(\Lambda)$. The groupoid model $\G$ for $C^*(\Lambda)$
\cite{KP} determines a Borel equivalence relation $\RR$ on the path
space~$\Lambda^\infty$. The unique probability measure $\mueq$ on $\Lambda^\infty$
induced by all KMS states of~$C^*(\Lambda)$ (see \cite[Proposition~8.1]{AHLRS2}) is
quasi-invariant with respect to $\RR$. The key result, Proposition~\ref{pmainp}, says
that $W^*(\RR)$ is isomorphic to the factor determined by any extremal KMS state of
$C^*(\Lambda)$; this isomorphism is noncanonical unless $\Lambda$ is aperiodic. We finish
the section by proving that the sub-relation $\RR^D$ obtained from $\RR$ as in the
preceding paragraph contains a still smaller relation~$\RR^\gamma$ which is an \'etale
topological equivalence relation whose $C^*$-algebra is the AF core of~$C^*(\Lambda)$.

In Section~\ref{sec:PF}, we develop a Frobenius analysis of strongly connected
higher-rank graphs. We investigate the group $\Pp_\Lambda \subseteq \Z^k$ generated by
the degrees $d(\lambda)$ of cycles in $\Lambda$. We show that there is a map $C :
\Lambda^0 \times \Lambda^0 \to \Z^k/\Pp_\Lambda$ such that $C(r(\lambda), s(\lambda)) =
d(\lambda) + \Pp_\Lambda$ for all $\lambda$. The key result of the section,
Proposition~\ref{properiod}, says that there is a strictly positive $p \in \Pp_\Lambda$
with the following property: $C(v,w) = 0$ if and only if there is a path of degree $p$
connecting $v$ to $w$. We also show that the relation $\sim$ given by $v \sim w$ if and
only if $C(v,w) = 0$ is an equivalence relation on $\Lambda^0$. We deduce that there
exists a natural free and transitive action of $\Z^k/\Pp_\Lambda$ on $\lms$, and that the
decomposition of $\R^{\Lambda^0}$ into direct summands indexed by $\lms$ is a system of
imprimitivity for the connectivity matrices $A_i$. We also deduce that~$\Pp_\Lambda$
always contains the periodicity group $\Per\Lambda$ of \cite{CKSS, AHLRS2}, with equality
if and only if $\Lambda$ is a simple $k$-dimensional cycle (see
Proposition~\ref{prp:P=Per}).

Finally, in Section~\ref{sec:typeclass}, we prove our main theorem. We show that the AF
core of $C^*(\Lambda)$ decomposes as a direct sum with summands indexed by $\lms$.
Corollary~\ref{cergodic} shows that the ergodic components of $\RR^\gamma$ are the sets
$X_\omega$ of infinite paths with range in $\omega \in \lms$. Each characteristic
function $\unb_{X_\omega}$ is a full projection in $W^*(\RR)$, so the type of $W^*(\RR)$
coincides with that of $W^*(\RR|_{X_\omega})$; we compute the latter using the results of
Section~\ref{sec:S-general}. We briefly discuss the relationship between our results and
Yang's, and show that the factorisation property in $\Lambda$ does not affect the factors
that arise from it. We conclude by applying our main theorem to a few illustrative
examples.

\section{Higher-rank graphs} \label{sprelim}

We denote by $\N$ the monoid $\{0,1,2,\dots\}$ of nonnegative integers under addition.
For an integer $k \ge 1$, we then regard $\N^k$ as a monoid with pointwise addition. The
canonical generators of $\N^k$ are denoted $e_i$, and for $n \in \N^k$ we write $n_i$ for
its $i$\textsuperscript{th} coordinate. We give $\N^k$ its natural partial order $m \le
n$ if and only if each $m_i \le n_i$ and $m \vee n$ denotes the coordinatewise maximum of
$m,n \in \N^k$.

A rank-$k$ graph, or a $k$-graph, is a small category $\Lambda$ together with an
assignment of a \emph{degree} $d(\lambda)\in\N^k$ to every morphism $\lambda\in\Lambda$
such that
\begin{enumerate}
\item $d(\lambda\mu)=d(\lambda)+d(\mu)$; and
\item whenever $d(\lambda)=m+n$, there is a unique factorisation $\lambda=\mu\nu$
    such that $d(\mu)=m$ and $d(\nu)=n$.
\end{enumerate}
Condition~(2) is often called the ``factorisation property.'' It implies in particular
that the only morphisms of degree $0$ are the identity morphisms.

The set of morphisms of degree $n\in\N^k$ is denoted by $\Lambda^n$. Its elements are
called paths of degree~$n$ in $\Lambda$. So $\Lambda^0$ is the set of identity morphisms;
we regard them interchangeably as paths of degree zero and as vertices. We also identify
$\Lambda^0$ with the set of objects of $\Lambda$ in the natural way, so that the codomain
and domain maps become functions $r, s : \Lambda \to \Lambda^0$. Throughout the paper we
consider only finite $k$-graphs, meaning that each $|\Lambda^n| < \infty$.

For $\mu \in \Lambda$ and $n \in \N^k$, denote by $\mu\Lambda^n$ the set of morphisms
$\mu\lambda$ such that $d(\lambda)=n$ and $s(\mu)=r(\lambda)$. The sets $\Lambda^n\nu$
and $\mu\Lambda^n\nu$ are defined similarly.

The \emph{connectivity matrices} $A_1, \dots, A_k\in\Mat_{\Lambda^0}(\N)$ of $\Lambda$
are given by
\[
A_i(v,w) = |v\Lambda^{e_i}w|.
\]
The factorisation property implies that the matrices $A_i$ pairwise commute. For $n \in
\N^k$, we define
\begin{equation}\label{pairwisecomm}\textstyle
A^n := \prod^k_{i=1} A_i^{n_i}.
\end{equation}
We then have $A^n(v,w)=|v\Lambda^nw|$ for all $v,w$, and $n \mapsto A^n$ is a semigroup
homomorphism. We write~$\rho(B)$ for the spectral radius of a square matrix $B$. Define
\[
\rho(\Lambda) := (\rho(A_1), \rho(A_2), \dots, \rho(A_k)) \in [0,\infty)^k.
\]
For $g \in \Z^k$ we write $\rho(\Lambda)^g$ for the product $\prod^k_{i=1}
\rho(A_i)^{g_i}$.

A finite $k$-graph $\Lambda$ is \emph{strongly connected} if $v\Lambda w \not= \emptyset$
for all $v,w \in \Lambda^0$. If there exists $p$ such that $v\Lambda^p w \not=\emptyset$
for all $v$ and $w$, then $\Lambda$ is called \emph{primitive}.

When working with strongly connected $k$-graphs, there is no loss of generality in
assuming that~$\Lambda^n$ is nonempty for every $n \in \N^k$, and it is then not
difficult to check (see \cite[Lemma~2.1]{AHLRS2} and the paragraph before it) that
\begin{equation}\label{eq:nosinks}
v\Lambda^n \not= \emptyset\text{ and } \Lambda^n v \not= \emptyset
    \quad\text{ for all $v \in \Lambda^0$ and $n \in \N^k$.}
\end{equation}
So each column and each row of each $A^n$ is nonzero. It then follows from
\cite[Corollary~4.2]{AHLRS2} and \cite[Lemma~A.1]{aHLRS2013} that each $\rho(A^n) \ge 1$
and that $n \mapsto \rho(A^n)$ is a homomorphism of $\N^k$ into the multiplicative
semigroup $[1,\infty)$.\label{pg:mpctve} Hence $\rho(A^n) = \rho(\Lambda)^n$ for all $n$.

The Toeplitz algebra $\Tt C^*(\Lambda)$ of the $k$-graph $\Lambda$ is the universal
$C^*$-algebra generated by elements $\{t_\lambda \mid \lambda \in \Lambda\}$ such that
\begin{itemize}
\item[(TCK1)] $\{t_v \mid v \in \Lambda^0\}$ is a family of mutually orthogonal
    projections;
\item[(TCK2)] $t_\mu t_\nu = t_{\mu\nu}$ whenever $s(\mu) = r(\nu)$;
\item[(TCK3)] $t_\mu^* t_\mu = t_{s(\mu)}$ for all $\mu$;
\item[(TCK4)] $t_v \ge \sum_{\mu \in v\Lambda^n} t_\mu t^*_\mu$ for all $v \in
    \Lambda^0$ and $n \in \N^k$; and
\item[(TCK5)] $t^*_\mu t_\nu = \sum_{\mu\alpha = \nu\beta \in \Lambda^{d(\mu) \vee
    d(\nu)}} t_\alpha t^*_\beta$ for all $\mu,\nu$.
\end{itemize}
The $C^*$-algebra $C^*(\Lambda)$ of $\Lambda$ is the quotient of $\Tt C^*(\Lambda)$ by
the ideal generated by $\{t_v - \sum_{\mu \in v\Lambda^n} t_\mu t^*_\mu \mid v
\in\Lambda^0, n \in \N^k\}$. It is universal for families $\{s_\lambda \mid
\lambda\in\Lambda\}$ satisfying (TCK1)--(TCK3) and
\begin{itemize}
\item[(CK)] $s_v = \sum_{\mu \in v\Lambda^n} s_\mu s^*_\mu$ for all $v,n$.
\end{itemize}

\section{The main result, and the factor types of KMS states on the Toeplitz algebra}\label{sec:mainthm}

Let $\Lambda$ be a strongly connected finite $k$-graph. Our main result is a
characterisation of the factor types of the extremal KMS states of $\Tt C^*(\Lambda)$ and
$C^*(\Lambda)$ studied in \cite{AHLRS2}. Each $r \in [0,\infty)^k$ determines an action
$\alpha^r\colon\R \to \Aut \Tt C^*(\Lambda)$ by $\alpha^r_t(t_\lambda) = e^{itr\cdot
d(\lambda)} t_\lambda$, and this descends to an action, also denoted~$\alpha^r$, on
$C^*(\Lambda)$. Corollary~4.6 of \cite{AHLRS2} shows that, up to rescaling, the only
value of $r$ for which~$\alpha^r$ admits a KMS state that factors through $C^*(\Lambda)$
is $r = \log\rho(\Lambda)$. We write an unadorned $\alpha$ for this dynamics, and call it
the ``preferred dynamics". Corollary~4.6 of \cite{AHLRS2} also shows that there are
$\alpha$-KMS$_\beta$ states for all $\beta \ge 1$, and the only ones that factor through
$C^*(\Lambda)$ occur at $\beta = 1$. Given a state $\phi$, we write $\pi_\phi$ for the
associated GNS representation.

\begin{theorem}\label{thm:main}
Let $\Lambda$ be a strongly connected finite $k$-graph. Let $\alpha$ be the preferred
dynamics and suppose that $\phi$ is an extremal $\alpha$-KMS$_\beta$ state of $\Tt
C^*(\Lambda)$.
\begin{enumerate}
\item\label{it:b>1} Suppose that $\beta > 1$. Then $\pi_\phi(\Tt C^*(\Lambda))''$ is
    the type~$\mathrm{I}_\infty$ factor.
\item\label{it:b=1} Suppose that $\beta = 1$, and write $\overline{\phi}$ for the
    corresponding KMS state of $C^*(\Lambda)$.
    \begin{enumerate}
    \item\label{it:b=1;r=1} If $\rho(\Lambda) = (1, \dots, 1)$, then
        $\pi_\phi(\Tt C^*(\Lambda))'' = \pi_{\overline{\phi}}(C^*(\Lambda))$ is
        the type~$\mathrm{I}_{|\Lambda^0|}$ factor.
    \item\label{it:b=1;r>1} Otherwise, let
    \[
        S := \{\rho(\Lambda)^{d(\mu) - d(\nu)} \mid \mu,\nu \in \Lambda\text{ are cycles}\},
    \]
    and let $\lambda := \sup\{s \in S \mid s < 1\}$. Then $\lambda \in (0,1]$ and
    $\pi_\phi(\Tt C^*(\Lambda))'' = \pi_{\overline{\phi}}(C^*(\Lambda))''$ is the
    injective type~$\mathrm{III}_\lambda$ factor.
    \end{enumerate}
\end{enumerate}
\end{theorem}

The rest of the paper mainly consists of the proof of Theorem~\ref{thm:main}, which will
be completed in Section~\ref{sec:typeclass}. Most of the work lies in
statement~(\ref{it:b=1;r>1}); in particular, (\ref{it:b>1}) is fairly straightforward,
and the proof works for any $\alpha^r$ and any $\beta$ with $\beta r > \log\rho(\Lambda)$
coordinatewise.

Theorem~6.1 of \cite{AHLRS1} describes KMS states of $\Tt C^*(\Lambda)$ as follows. Take
$r \in [0,\infty)^k$ and $\beta \in \R$ such that $\beta r > \log\rho(\Lambda)$
coordinatewise. Then for each $v$, the series $\sum_{\mu\in\Lambda v} e^{-\beta r\cdot
d(\mu)}$ converges to some $y_v \geq 1$. Set $y = (y_v)_{v \in \Lambda^0}$. For each
$\epsilon \in [0,\infty)^{\Lambda^0}$ such that $\epsilon \cdot y = 1$, define $\Delta
\colon \Lambda \to \R^+$ by $\Delta_\lambda = e^{-\beta r \cdot
d(\lambda)}\epsilon_{s(\lambda)}$. Let $\{h_\lambda \mid \lambda \in \Lambda\}$ be the
orthonormal basis for $\ell^2(\Lambda)$, and $\pi_S : \Tt C^*(\Lambda) \to
\Bb(\ell^2(\Lambda))$ the path-space representation $\pi_S(t_\lambda) h_\mu =
\delta_{s(\lambda), r(\mu)} h_{\lambda\mu}$ \cite[Example~7.7]{RS}. Then the formula
$\varphi_\epsilon(a) = \sum_{\lambda\in\Lambda} \Delta_\lambda (\pi_S(a)h_\lambda |
h_\lambda)$ defines an $\alpha$-KMS$_\beta$ state $\varphi_\epsilon$ of $\Tt
C^*(\Lambda)$. Moreover, putting $m^\epsilon := \prod_{i=1}^k(1-e^{-\beta
r_i}A_i)^{-1}\epsilon$, we have $\varphi_\epsilon(t_\mu t_\nu^*)=
\delta_{\mu,\nu}e^{-\beta r\cdot d(\mu)}m^\epsilon_{s(\mu)}$ for all $\mu,\nu$. The map
$\epsilon \mapsto \varphi_\epsilon$ from $\Sigma_\beta = \{ \epsilon \in
[0,\infty)^{\Lambda^0} \mid \epsilon \cdot y = 1\}$ to the simplex of
$\alpha$-KMS$_\beta$ states of $\Tt C^*(\Lambda)$ is an affine isomorphism. The simplex
$\Sigma_\beta$ is the closed convex hull of $\{y_v^{-1}\delta_v \mid v \in \Lambda^0\}$.

\begin{proposition}\label{prp:toeplitzfactors}
Let $\Lambda$ be a strongly connected finite $k$-graph. Suppose that $r \in [0,\infty)^k$
and $\beta > 0$ satisfy $\beta r > \log\rho(\Lambda)$ coordinatewise. Fix $v \in
\Lambda^0$ and let $\epsilon := y_v^{-1} \delta_v \in [0,\infty)^{\Lambda^0}$. Let
$\pi_{\epsilon}:=\pi_{\varphi_\epsilon}$ be the GNS representation associated to the
KMS$_\beta$ state $\varphi_\epsilon$. Then $\pi_\epsilon(\Tt C^*(\Lambda))''$ is a factor
of type I$_\infty$, and $\pi_\epsilon(t_\mu) \mapsto \pi_S(t_\mu)|_{\ell^2(\Lambda v)}$
determines a von Neumann algebra isomorphism $\pi_\epsilon(\Tt C^*(\Lambda))'' \cong
\Bb(\ell^2(\Lambda v))$.
\end{proposition}
\begin{proof}
By \cite[Theorem~6.1]{AHLRS1}, the state $\varphi_\epsilon$ is an extremal point in the
simplex of KMS$_\beta$ states. Hence $\pi_\epsilon(\Tt C^*(\Lambda))''$ is a factor (see,
for example \cite[Theorem~5.3.30(3)]{BR1997a}).

Let $q_v := \prod_{i=1}^k (t_v - \sum_{\gamma \in v\Lambda^{e_i}} t_\gamma t_\gamma^*)$.
We have $\pi_S(q_v) h_v = h_v$, and for $\mu \in \Lambda\setminus v\Lambda$ we have $q_v
\le t_v \perp t_{r(\mu)}$ giving $\pi_S(q_v) h_\mu = 0$. For $\mu \in v\Lambda \setminus
\{v\}$, we have $\mu = \mu_1\mu'$ for some $\mu_1 \in \bigcup_i v\Lambda^{e_i}$, and then
$\pi_S(t_v - t_{\mu_1} t^*_{\mu_1}) h_\mu = 0$, forcing $\pi_S(q_v)h_\mu = 0$. Hence
$\pi_S(q_v)$ is the projection onto $\C h_v$, and $\pi_S(q_v \Tt C^*(\Lambda) q_v) = \C
\pi_S(q_v)$. Since $\pi_S$ is faithful by \cite[Corollary~7.7]{RS}, we then have $q_v \Tt
C^*(\Lambda) q_v = \C q_v$, and hence $\pi_\epsilon(q_v) \pi_\epsilon(\Tt C^*(\Lambda))''
\pi_\epsilon(q_v) = \C \pi_\epsilon(q_v)$. Thus $\pi_\epsilon(q_v)$ is either a minimal
projection or zero. Since $\epsilon = y_v^{-1}\delta_v$, we have
\begin{equation}\label{eq:phieps}\textstyle
\varphi_\epsilon(q_v)
    = \sum_{\lambda\in\Lambda} \Delta_\lambda (\pi_S(q_v)h_\lambda \mid h_\lambda)
    = \Delta_v
    = e^{-\beta r \cdot d(v)}\epsilon_v
    = y^{-1}_v > 0.
\end{equation}
Hence $\pi_\epsilon(q_v) \not= 0$, and so $\pi_\epsilon(\Tt C^*(\Lambda))''$ has a
minimal projection $\pi_\epsilon(q_v)$, and is therefore of type~I.

For $\lambda \in \Lambda v$, let $\xi_\lambda$ denote the class of $\sqrt{y_v}t_\lambda
q_v$ in the GNS space $\Hh_\epsilon$ of $\varphi_\epsilon$. For $\lambda,\mu \in \Lambda
v$, we have
\begin{align*}
\big(\xi_\lambda \mid \xi_\mu\big)_{\Hh_\epsilon}
    = y_v\varphi_\epsilon (q_v t_\mu^* t_\lambda q_v)
    &= y_v\sum_{(\gamma,\gamma') \in \Lmin(\mu,\lambda)} \varphi_\epsilon(q_v t_\gamma t_{\gamma'}^* q_v)\\
    &= y_v\sum_{(\gamma,\gamma') \in \Lmin(\mu,\lambda)} \varphi_\epsilon(\delta_{v,\gamma} \delta_{v,\gamma'}q_v)\\
    &= y_v\delta_{\lambda,\mu} \delta_{s(\lambda), v}\,\varphi_\epsilon(q_v)
    = \delta_{\lambda,\mu} \delta_{s(\lambda), v}
\end{align*}
by~\eqref{eq:phieps}. So $\{\xi_\lambda \mid \lambda \in \Lambda v\}$ is an infinite
orthonormal set in $\Hh_\epsilon$. Define $\Hh := \clsp\{\xi_\lambda \mid \lambda \in
\Lambda v\}$. For $\lambda \in \Lambda v$ and $\mu, \nu \in \Lambda$, we have
\[
\pi_\epsilon(t_\mu t^*_\nu) \xi_\lambda
    = \sqrt{y_v} [t_\mu t^*_\nu t_\lambda q_v]
    = \sqrt{y_v}\Big[\sum_{(\gamma,\gamma') \in \Lmin(\nu,\lambda)} t_{\mu\gamma} (q_v t_{\gamma'})^*\Big].
\]
We have $\pi_S(t_{\gamma'})\ell^2(\Lambda) = \clsp\{h_{\gamma'\lambda} \mid \lambda \in
s(\gamma')\Lambda\}$, which is orthogonal to $\pi_S(q_v)\ell^2(\Lambda) = \C \delta_v$
unless $\gamma' = v$. Since $\pi_S$ is injective, we deduce that $\pi_\epsilon(t_\mu
t^*_\nu) \xi_\lambda$ is zero unless $\lambda = \nu\lambda'$ and $\Lmin(\nu,\lambda) =
\{(\lambda',v)\}$, giving $\pi_\epsilon(t_\mu t^*_\nu) \xi_\lambda = \xi_{\mu\lambda'}
\in \Hh$. So $\Hh$ is invariant for $\pi_\epsilon(\Tt C^*(\Lambda))$, and hence also for
its double commutant. Since $\pi_\epsilon(\Tt C^*(\Lambda))''$ is a factor, it follows
that the restriction map $T \to T|_\Hh$ is a von Neumann algebra isomorphism.

For $\mu,\nu \in \Lambda v$, we have $\pi_\epsilon(t_\mu q_v t^*_\nu) \xi_\lambda =
\delta_{\nu,\lambda} \xi_\mu$, and so $\pi_\epsilon(t_\mu q_v t^*_\nu)$ is the rank-one
operator $\theta_{\xi_\mu, \xi_\nu}$. So $\Kk(\Hh) \subseteq \pi_\epsilon(\Tt
C^*(\Lambda))$, giving $\pi_\epsilon(\Tt C^*(\Lambda))''|_{\Hh} = \Bb(\Hh)$. The formula
$U \colon \xi_\lambda \mapsto h_\lambda$ defines a unitary isomorphism $\Hh \cong
\ell^2(\Lambda v)$, and we have $U\pi_\epsilon(t_\mu)|_{\Hh} U^* = \pi_S(t_\mu)$.
\end{proof}

\section{The Connes invariant of the von Neumann algebra of an equivalence
relation}\label{sec:S-general}

Our analysis of the types of extremal KMS states on the $C^*$-algebra of a $k$-graph will
rely on identifying the associated factor with the von Neumann algebra of the equivalence
relation determined by the $k$-graph groupoid. Computing the Connes invariant $S$ of a
factor is in general a difficult problem; but it simplifies drastically in the presence
of a faithful normal state $\varphi$ with factorial centraliser. In this instance,
Connes' result \cite[Th\'eor\`eme~3.2.1]{Connes:ASENS73} implies that $S$ is equal to the
spectrum of the modular operator for $\varphi$. In this section we describe what this
result says for von Neumann algebras of equivalence relations; this result is surely
known, but we provide a proof as we were unable to find a reference.

Let us review Feldman and Moore's construction of the von Neumann algebra $W^*(\QQ)$ of a
countable Borel equivalence relation $\QQ$ on a space $X$ and a quasi-invariant measure
$\mu$ on $X$ (see \cite[Section~2]{FM1} and \cite[Section~2]{FM2}). Recall that a Borel
equivalence relation $\QQ$ on a standard Borel space $X$ is said to be \emph{countable}
if $\{y \mid (x,y) \in \QQ\}$ is countable for each $x\in X$. A Borel measure $\mu$ on
$X$ is \emph{quasi-invariant} for~$\QQ$ if, whenever $\mu(A) = 0$, the $\QQ$-saturation
$\QQ(A) := \bigcup_{x \in A} \{y \mid (x,y) \in \QQ\}$ of $A$ is also $\mu$-null.  Equip
$\QQ$ with the left counting measure $\nu$:
$$
\nu(C)=\int_X\big|\{y\mid (x,y)\in C\}\big|\,d\mu(x)
    \quad \text{for Borel $C\subseteq \QQ$}.
$$
(Unlike Feldman and Moore, we work with left counting measure, not right, as this is
consistent with Renault's representation theory of groupoids.) Identifying the diagonal
of $\QQ$ with $X$ via $(x,x) \mapsto x$, the restriction of $\nu$ to the diagonal
coincides with~$\mu$. A Borel subset $A$ of $\QQ$ is called a \emph{bisection} if the
projection maps $(x,y)\mapsto x$ and $(x,y)\mapsto y$ are injective on $A$. Consider the
$*$-algebra $\C[\QQ]$ of functions $f\in L^\infty(\QQ,\nu)$ supported on finitely many
Borel bisections of $\QQ$, under the convolution
$$
(f_1*f_2)(x,z)=\sum_{(x,y) \in \QQ} f_1(x,y)f_2(y,z)
$$
and involution $f^*(x,y)=\overline{f(y,x)}$. Write $D$ for the Radon--Nikodym cocycle on
$\QQ$ determined by $\mu$. Then $\C[\QQ]$ has a representation~$\pi$ on $L^2(\QQ,d\nu)$
given by
\begin{equation}\label{eq:pirep}
(\pi(f)\xi)(x,z)=\sum_{(x,y) \in \QQ}D(x,y)^{-1/2}f(x,y)\xi(y,z),
\end{equation}
and by definition, $W^*(\QQ) = \pi(\C[\QQ])''$. The characteristic function
$\unb_{\{(x,x) \mid x \in X\}}$ is a cyclic separating vector for $W^*(\QQ)$, so the
formula
\begin{equation}\label{eq:varphistate}
\varphi(f)=\int_X f(x,x)d\mu(x)\quad\text{for } f\in\C[\QQ]
\end{equation}
defines a faithful normal state $\varphi$ of $W^*(\QQ)$.

\begin{proposition}\label{prp:centralizer}
Let $\QQ$ be a countable measurable equivalence relation on a standard Borel space
$(X,\mu)$ with Radon--Nikodym cocycle $D$. Let $\QQ^D$ denote the finer equivalence
relation
$$
x\sim_{\QQ^D}y\text{ if and only if } x\sim_{\QQ}y\text{ and } D(x,y)=1.
$$
Identify $W^*(\QQ^D)$ with the strong-operator closure of the subalgebra $$\{\pi(f)\mid f
\in \C[\QQ],\ \operatorname{supp}(f) \subseteq \QQ^D\}\subset W^*(\QQ).
$$
Then $W^*(\QQ^D)$ is equal to the centraliser of the state $\varphi$ on $W^*(\QQ)$ defined by~\eqref{eq:varphistate}.
\end{proposition}
\bp%
Let $M := W^*(\QQ)$ and $N := W^*(\QQ^D) \subseteq M$. Let $\nu$ denote left counting
measure on $\QQ$, and let $\xi := \unb_{\{(x,x) \mid x \in X\}} \in L^2(\QQ, d\nu)$; so
$\varphi$ is the vector state associated to $\xi$. Let $\Delta$ be the modular operator for
$\varphi$. Then \cite[Proposition~2.8]{FM2} implies that $\Delta$ is given by multiplication
by $D$ on $L^2(\QQ,d\nu)$. Hence the eigenspace of $\Delta$ corresponding to the
eigenvalue $1$ is $L^2(\QQ^D, d\nu)$, which is precisely $\overline{N\xi}$.

Hence $N$ is contained in the centraliser $M_\varphi$ of $\varphi$ and
$\overline{N\xi}=\overline{M_\varphi\xi}$. In particular, $N$ is invariant under the
modular group. So  there is a unique $\varphi$-preserving conditional expectation
$E\colon M\to N$. Let $p$ be the projection onto $\overline{N\xi}$. For $a\in M$ we have
$E(a)\xi=pa\xi$. Hence, if $a\in M_\varphi$, then $E(a)\xi=pa\xi=a\xi$, so
$a=E(a)\in N$. Thus $N=M_\varphi$.%
\ep

\begin{corollary} \label{ceqConnes}
Let $\QQ$, $D : \QQ \to (0,\infty)$ and $\QQ^D$ be as in
Proposition~\ref{prp:centralizer}. If $\QQ^D$ is ergodic on $(X,\mu)$, then the Connes
invariant $S(W^*(\QQ))$ is the essential range of $D$.
\end{corollary}
\bp%
Since $\QQ^D$ is ergodic, Proposition~\ref{prp:centralizer} implies that the centraliser
of $\varphi$ is a factor. So \cite[Th\'eor\`eme~3.2.1]{Connes:ASENS73} shows that
$S(W^*(\QQ))$ is equal to the spectrum of the operator of multiplication by~$D$ on
$L^2(\QQ, d\nu)$. Since this spectrum is exactly the essential range of $D$, the result follows.%
\ep

\begin{remark}
An alternative proof of Corollary~\ref{ceqConnes} is as follows. First
apply~\cite[Proposition 2.11]{FM2}, to see that $S(W^*(\QQ))$ is equal to the ratio set
of $D$, and then argue directly from its definition that this coincides with the
essential range when~$\QQ^D$ is ergodic.
\end{remark}

In our application the equivalence relation will arise as the orbit equivalence relation
defined by the action of a groupoid on its unit space. For the convenience of the reader,
we record the following useful relation between the two Radon--Nikodym cocycles arising
in this situation.

\begin{lemma} \label{lRN}
Let $\G$ be a second countable \'etale groupoid and let $\mu$ be a quasi-invariant
measure on~$\G^0$ with Radon--Nikodym cocycle $c\colon \G\to(0,\infty)$. Consider the
orbit equivalence relation $\QQ$ defined by~$\G$ on~$\G^0$. Then the measure $\mu$ is
quasi-invariant with respect to $\QQ$, and if $D\colon  \QQ\to(0,\infty)$ is the
corresponding Radon--Nikodym cocycle, then there exists a $\QQ$-invariant $\mu$-conull
Borel subset $X\subseteq\G^0$ such that $D(x,y)=c(g)$ for all $(x,y)\in\QQ\cap(X\times
X)$ and $g\in\G^x_y:=\{g\in\G\mid x=r(g), \ y=s(g)\}$.
\end{lemma}

\bp Choose a countable cover $\{V_n\}^\infty_{n=1}$ of $\G$ by open bisections; that is,
open sets $V_n$ on which~$r$ and~$s$ are injective. For each $n$ define $T_n : r(V_n) \to
s(V_n)$ by $T_n(r(g)) = s(g)$ for $g \in V_n$. By quasi-invariance of $\mu$ for $\G$, the
$T_n$ preserve the measure class of $\mu$ and
\[\textstyle
\frac{d(T_n)_*\mu}{d\mu}(s(g)) = c(g)\quad\text{ for $g\in V_n$.}
\]
So if $\mu(A) = 0$, then $\mu\big(\bigcup_n T_n^{-1}(A)\big) = 0$ too. Since $\bigcup_n
T_n^{-1}(A)$ is precisely the $\QQ$-saturation of $A$, the measure $\mu$ is
quasi-invariant for $\QQ$.

By Proposition~2.2 of \cite{FM1} applied to the Borel isomorphism $T_n : r(V_n) \to
s(V_n)$ with graph in $\QQ$, there is a $\mu$-null $Y_n \subseteq r(V_n)$ such that for
$g \in V_n \setminus r^{-1}(Y_n)$, we have
$$
D(r(g), s(g)) = D(T_n^{-1}(s(g)),s(g)) = \frac{d (T_n)_*\mu}{d\mu}(s(g)) = c(g).
$$
Now $Y := \bigcup_n Y_n$ is $\mu$-null, so its $\QQ$-saturation $\QQ(Y)$ is $\mu$-null
and $\QQ$-invariant. Thus $X := \G^0 \setminus \QQ(Y)$ suffices. \ep

\begin{remark}\label{rRN}
As a byproduct we see that for $\mu$-a.e.~$x\in\G^0$ we have $c(g)=1$ for all
$g\in\G^x_x$, since $D(x,x)=1$ for $\mu$-a.e.~$x$.
\end{remark}

\section{Equivalence relations and KMS states of \texorpdfstring{$C^*(\Lambda)$}{C*(Lambda)}} \label{sequiv}

We will now use the groupoid picture for the $C^*$-algebra of a strongly connected finite
$k$-graph. We briefly recount the construction of the groupoid $\G$ associated to a
$k$-graph $\Lambda$ and refer to \cite{KP} or Section~12 of \cite{AHLRS2} for more
details.

The unit space of $\G$ is the space $\Lambda^\infty$ of infinite paths in $\Lambda$,
which is defined as follows. Let~$\Omega_k$ denote the $k$-graph with objects $\N^k$,
morphisms $\{(m,n) \in \N^k \times \N^k \mid m \le n\}$, structure maps $r(m,n) = m$,
$s(m,n) = n$ and $d(m,n) = n-m$, and composition $(m,n)(n,p) = (m,p)$. Then an infinite
path $x\in\Lambda^\infty$ is a degree-preserving functor $\Omega_k\to\Lambda$.

For $n\in\N^k$ denote by $\sigma^n$ the shift on $\Lambda^\infty$ corresponding to $n$,
so $\sigma^n(x)(p,q)=x(p+n,q+n)$. Then, as a set,
$$
\G=\{(x,g,y)\in\Lambda^\infty\times\Z^k\times\Lambda^\infty\mid \sigma^{g+n}(x)=\sigma^n(y)
    \text{ for some $n \in \N^k$}\}.
$$
The source and range maps are $s(x,g,y) = y$ and $r(x,g,y) = x$, and composition is
$$
(x,g,y)(y,h,z)=(x,g+h,z).
$$
For $x \in \Lambda^\infty$ and $\mu \in \Lambda$ with $r(x):=x(0,0) = s(\mu)$, there is a
unique $\mu x \in \Lambda^\infty$ such that $(\mu x)(0,d(\mu)) = \mu$ and
$\sigma^{d(\mu)}(\mu x) = x$. The sets $Z(\mu,\nu) = \{(\mu x, d(\mu) - d(\nu),\nu x)
\mid x\in\Lambda^\infty,\ s(\mu)=s(\nu)=r(x)\}$ indexed by pairs $\mu,\nu \in \Lambda$
form a basis of compact open sets for the topology.

There is an isomorphism $C^*(\Lambda) \cong C^*(\G)$ that carries each $s_\lambda$
to~$\unb_{Z(\lambda, s(\lambda))}$. The preferred dynamics~$\alpha$ on
$C^*(\Lambda)=C^*(\G)$ corresponds to the cocycle $c\colon\G\to\R$ given by
$$
c(x,g,y)=g\cdot\log\rho(\Lambda).
$$
That is, for $f\in C_c(\G)\subset C^*(\G)$ we have
$$
\alpha_t(f)(x,g,y)=e^{it c(x,g,y)}f(x,g,y) = \rho(\Lambda)^{itg} f(x,g,y)
$$
for all $t\in\R$ and $(x,g,y)\in\G$.

We now briefly recap the description of the $\alpha$-KMS states of $C^*(\Lambda)$
from~\cite{AHLRS2}. By \cite[Proposition~8.1]{AHLRS2}, there is a unique Borel
probability measure $\mueq$ on $\Lambda^\infty$ with Radon--Nikodym cocycle $e^{-c}$. Let
\[
\Per \Lambda := \{m - n \mid m,n \in \N^k, \sigma^m(x) = \sigma^n(x)
    \text{ for all }x \in \Lambda^\infty\}.
\]
Then $\Per\Lambda$ is a subgroup of $\Z^k$, and \cite[Theorem 7.1]{AHLRS2} describes a
one-to-one correspondence between the extremal $\alpha$-KMS$_1$-states on $C^*(\G)$ and
the characters of $\Per\Lambda$: the state $\varphi_\chi$ corresponding to a character
$\chi \in (\Per\Lambda)\widehat{\;}$ is given by
\begin{equation}\label{defphichi}
\varphi_\chi(f)=\int_{\Lambda^\infty}\sum_{g\in\Per\Lambda}\chi(g)f(x,g,x)\,d\mueq(x)
    \quad\text{ for $f\in C_c(\G)$}.
\end{equation}
Let $\pi_{\varphi_\chi}$ denote the GNS representation associated to $\varphi_\chi$. Our
goal is to understand the corresponding von Neumann algebra
$\pi_{\varphi_\chi}(C^*(\G))''$. The following technical result from~\cite{AHLRS2} will
play a key role.

\begin{proposition}{\cite[Lemma 12.1]{AHLRS2}} \label{pAHLRS}
Let $\Lambda$ be a strongly connected finite $k$-graph, and $\G$ be the associated
groupoid. For $\meae$ $x\in\Lambda^\infty$, we have $\G^x_x = \{x\} \times \Per\Lambda
\times \{x\}$.
\end{proposition}

Note that by Remark~\ref{rRN}, quasi-invariance of $\mueq$ immediately gives $\G^x_x
\subseteq \{(x,g,x) \mid \rho(\Lambda)^g = 1\}$ for $\meae$~$x$. But
Proposition~\ref{pAHLRS} says much more for $k \ge 2$, since the inclusion
\begin{equation} \label{eper}
\Per\Lambda\subseteq\{m-n\mid m,n\in\N^k, A^m=A^n\},
\end{equation}
proved in \cite[Remark 7.2]{AHLRS2}, shows that $\Per\Lambda$ is generally an infinite
index subgroup of $\{g \mid \rho(\Lambda)^g = 1\}$.

Consider the orbit equivalence relation $\RR$ on $\Lambda^\infty$ defined by $\G$, so
$$
x\sim_\RR y\text{ if and only if } \sigma^m(x)=\sigma^n(y)\text{ for some } m,n\in\N^k.
$$
By Lemma~\ref{lRN}, the measure $\mueq$ is quasi-invariant with respect to $\RR$, and the
corresponding Radon--Nikodym cocycle $D$ can be described as follows. There is an
$\RR$-invariant $\mueq$-conull set $X \subseteq \Lambda^\infty$ such that for all $x, y
\in X$ and $m,n \in \N^k$ satisfying $\sigma^m(x) = \sigma^n(y)$, we have
\begin{equation}\label{RadNik}
D(x,y) = e^{-c(x,m-n,y)} = \rho(\Lambda)^{n-m}.
\end{equation}

From now on we view $\RR$ as a measurable equivalence relation on
$(\Lambda^\infty,\mueq)$. We write $\nueq$ for the left counting measure on $\RR$
obtained from $\mueq$ as in Section~\ref{sec:S-general}, and write $\pi : \C[\RR] \to
\Bb(L^2(\RR, d\nueq))$ for the representation~\eqref{eq:pirep}, so that $W^*(\RR) =
\pi(\C[\RR])''$.

The following result underpins our computation of the factor type of the extremal
$\alpha$-KMS$_1$ states described in~\eqref{defphichi}. The idea for the isomorphism in
part~(\ref{it:vNA iso}) comes from \cite[Remark~2.5]{N}.

\begin{proposition} \label{pmainp}
Let $\Lambda$ be a strongly connected finite $k$-graph and $\G$ be the associated
groupoid. Let~$\chi$ be a character of $\Per\Lambda$,  let $\varphi_\chi$ be the extremal
KMS state of $C^*(\Lambda)$ defined in~\eqref{defphichi}, and let $\tilde\chi$ be a
character of $\Z^k$ extending $\chi$. Then
\begin{enumerate}
\item \label{it:homomorphism}  there is a $*$-homomorphism $\Phi_{\tilde\chi}\colon
    C_c(\G)\to \C[\RR]$ such that
$$
\Phi_{\tilde\chi}(f)(x,y)=\sum_{(x,g,y) \in \G}\tilde\chi(g)f(x,g,y)\quad\text{for }f\in C_c(\G);
$$
\item\label{it:vNA iso} the map $\pi_{\varphi_\chi}(f) \mapsto
    \pi(\Phi_{\tilde\chi}(f))$ for $f\in C_c(\G)$ extends uniquely to a von Neumann
    algebra isomorphism of $\pi_{\varphi_\chi}(C^*(\G))'' $ onto $W^*(\RR)$;
\item\label{it:Rfactor} $W^*(\RR)$ is a  factor and the equivalence relation $\RR$ on
    $(\Lambda^\infty,\mueq)$ is ergodic.
\end{enumerate}
\end{proposition}
\bp%
Part~(\ref{it:homomorphism}): Since $f \in C_c(\G)$, the formula for
$\Phi_{\tilde\chi}(f)$ has only finitely many nonzero terms, so the series converges, and
the function $\Phi_{\tilde\chi}(f)$ is supported on finitely many Borel bisections of
$\RR$. It is clear that $\Phi_{\tilde\chi}(f)$ is Borel. Thus $\Phi_{\tilde\chi}(f)$
belongs to~$\C[\RR]$. Direct calculation shows that $\Phi_{\tilde\chi}$ is a
$*$-homomorphism.

Part~(\ref{it:vNA iso}): Proposition~\ref{pAHLRS} implies that for $f \in C_c(\G)$, we
have
$$
\Phi_{\tilde\chi}(f)(x,x)=\sum_{g\in\Per\Lambda}\chi(g)f(x,g,x)
    \quad\text{for $\meae$ $x \in \Lambda^\infty$.}
$$
Hence the state $\varphi$ on $W^*(\RR)$ defined by~\eqref{eq:varphistate} satisfies
$\varphi\circ\Phi_{\tilde\chi}=\varphi_\chi$. Since $\Phi_{\tilde\chi}(C_c(\G))$ is
strong-operator dense in $W^*(\RR)$, we can identify $\pi\circ\Phi_{\tilde\chi}$ with the
GNS-representation of $\varphi_\chi$. Hence $\pi_{\varphi_\chi}(C^*(\G))''\cong
W^*(\RR)$.

Part~(\ref{it:Rfactor}): Since $\varphi_\chi$ is an extremal KMS state,
\cite[Theorem~5.3.30(3)]{BR1997a} implies that $W^*(\RR) \cong
\pi_{\varphi_\chi}(C^*(\Lambda))''$ is a factor. Proposition 2.9(2) of \cite{FM2} implies
that $\RR$ is ergodic.%
\ep

Equation~\eqref{RadNik} shows that the finer equivalence relation $\RR^D$ of
Proposition~\ref{prp:centralizer} for $\QQ = \RR$ and $\mu = \mueq$ is given, up to a set
of measure $0$, by
\begin{equation}\label{eq:RD}
x\sim_{\RR^D}y\text{ if and only if }\sigma^m(x)=\sigma^n(y)
    \text{ for some } m,n\in\N^k\text{ with } \rho(\Lambda)^{m-n} = 1.
\end{equation}
Corollary~\ref{ceqConnes} leads us to analyse this equivalence relation $\RR^D$. To do
this, it will help to consider an even finer equivalence relation.

Recall that $C^*(\Lambda)$ carries a canonical action $\gamma$ of $\T^k$, and that
$C^*(\Lambda)^\gamma$ denotes the fixed-point algebra for $\gamma$ (see \cite{KP}).

\begin{lemma}\label{lem:Rgamma}
Let $\Lambda$ be a strongly-connected finite $k$-graph and $\G$ be the associated
groupoid. Let $\RR^\gamma$ denote the relation
$$
x\sim_{\RR^\gamma} y\text{ if and only if }
    \sigma^n(x)=\sigma^n(y)\text{ for some } n\in\N^k,
$$
regarded as a subgroupoid of the measurable equivalence relation $\RR^D$. Then, using the
isomorphism $(x,y) \mapsto (x,0,y)$ of $\RR^\gamma$ onto the kernel $\G_0$ of the
canonical $\Z^k$-valued cocycle $(x,m,y) \mapsto m$ on $\G$, we can consider $\RR^\gamma$
as an \'etale topological equivalence relation, and $C^*(\RR^\gamma) \cong
C^*(\Lambda)^\gamma$.
\end{lemma}
\bp%
The isomorphism $(x,y) \mapsto (x,0,y)$ of $\RR^\gamma$ onto $\G_0$ is Borel, so it can
be used to define on~$\RR^\gamma$ the structure of an \'etale groupoid. The isomorphism
$C^*(\Lambda) \cong C^*(\G)$ carries the gauge action to the action given by
$\beta_z(f)(x,g,y) = z^g f(x,g,y)$ for $f \in C_c(\G)$ (see the proof of
\cite[Corollary~3.5(i)]{KP}). So it carries the fixed-point algebra
$C^*(\Lambda)^\gamma$ to the completion of the functions on $\G$ supported on~$\G_0$. Hence $C^*(\RR^\gamma) \cong C^*(\Lambda)^\gamma$.%
\ep

Let $\chi$ be a character of $\Per\Lambda$, and let $\Phi_{\tilde\chi}$ be as in
Proposition~\ref{pmainp}(\ref{it:homomorphism}). Lemma~\ref{lem:Rgamma} implies that
$W^*(\RR^\gamma)\subseteq W^*(\RR)$ is the von Neumann algebra generated by the image
under (the extension of) $\pi\circ \Phi_{\tilde\chi}$ of $C^*(\Lambda)^\gamma$. Note in
passing that the gauge action need not pass from $C^*(\Lambda)$ to~$W^*(\RR)$; it extends
precisely when $\Per\Lambda=0$.

\section{Frobenius analysis of strongly connected higher-rank graphs}\label{sec:PF}

In this section we analyse strongly connected finite $k$-graphs that are not necessarily
primitive to obtain a further refinement of the Perron--Frobenius theory for them
developed in \cite[Lemma~4.1]{KP2} and \cite[Corollary~4.2]{AHLRS2}. Generalisations of
Perron--Frobenius theory have been studied by a number of authors (see for example
\cite{GT} and the references therein), so part of the content of this section may be
known to experts. We give a self-contained presentation since we could not find exactly
what we need in the literature.

\begin{lemma}\label{lem:Pp} Let $\Lambda$ be a strongly connected $k$-graph. For
each vertex $v\in \Lambda^0$ let $\Pp_v^+ := d(v\Lambda v) \subseteq \N^k$. Then
$\Pp_v^+$ is a subsemigroup of $\N^k$ and $\Pp_v^+ - \Pp_v^+$ a subgroup of $\Z^k$. Let
\[
\Pp_\Lambda := \{d(\mu) - d(\nu) \mid \mu,\nu\text{ are cycles in $\Lambda$}\}.
\]
Then $\Pp_\Lambda = \Pp_v^+ - \Pp_v^+$ for every $v \in \Lambda^0$. We call $\Pp_\Lambda$
the \emph{group of periods} of $\Lambda$.
\end{lemma}
\begin{proof}
Fix $v \in \Lambda^0$. The set $\Pp_v^+$ is a semigroup because $d$ carries composition
to addition, and then $\Pp_v^+ - \Pp_v^+$ is obviously a subgroup of $\Z^k$. We clearly
have $\Pp_v^+ - \Pp_v^+ \subseteq \Pp_\Lambda$ for each $v$. For the reverse inclusion,
fix cycles $\mu,\nu \in \Lambda$ and $v \in \Lambda^0$. Since $\Lambda$ is strongly
connected, there exist $\lambda \in r(\nu)\Lambda r(\mu)$, $\lambda' \in r(\mu)\Lambda
r(\nu)$, $\eta \in v \Lambda r(\nu)$ and $\eta' \in r(\nu) \Lambda v$. Hence
\[
d(\mu) - d(\nu)
    = d(\eta\lambda\mu\lambda'\eta') - d(\eta\nu\lambda\lambda'\eta')
        \in \Pp_v^+ - \Pp_v^+.\qedhere
\]
\end{proof}

When $k=1$, the group $\Pp_\Lambda$ is the subgroup of $\Z$ generated by the classical
period of the directed graph $(\Lambda^0, \Lambda^1, r, s)$ (see, for example,
\cite{S,Kitchens}).

\begin{remark}\label{rmk:Pp dependence}
Since $|v\Lambda^n w| = A^n(v,w) = \big(\prod^k_{i=1} A_i^{n_i}\big)(v,w)$ for all
$v,w,n$, the group $\Pp_\Lambda$ depends only on the connectivity matrices of $\Lambda$
and is independent of the factorisation property. By contrast, for $k\ge2$ the group
$\Per\Lambda$ is not determined by the $A_i$ alone (see \cite{DY}).
\end{remark}

We now establish a number of properties of $\Pp_\Lambda$ that we will need in order to
compute the types of the KMS states of $C^*(\Lambda)$.

\begin{lemma}\label{twopaths}
Let $\Lambda$ be a strongly connected $k$-graph.  Then $d(\lambda) - d(\mu) \in
\Pp_\Lambda$ whenever $\lambda,\mu \in \Lambda$ satisfy $r(\lambda) = r(\mu)$ and
$s(\lambda) = s(\mu)$. In particular, there is a function $C : \Lambda^0 \times \Lambda^0
\to \Z^k/\Pp_\Lambda$ such that $C(r(\lambda), s(\lambda)) = d(\lambda) + \Pp_\Lambda$
for all $\lambda \in \Lambda$. For $u,v,w \in \Lambda^0$, we have
$$
C(u,v) + C(v,w) = C(u,w),\quad C(u,u) = 0,\quad\text{ and }\quad C(u,v) = -C(v,u).
$$
In particular, there is an equivalence relation $\sim$ on $\Lambda^0$ such that $v \sim
w$ if and only if $C(v,w) = 0$.
\end{lemma}
\begin{proof}
Fix $\lambda,\mu$ with $r(\lambda) = r(\mu)$ and $s(\lambda) = s(\mu)$. Choose $\nu \in
s(\lambda)\Lambda r(\lambda)$, so $\mu\nu$ and $\lambda\nu$ are cycles. Then $d(\lambda)
- d(\mu) = d(\lambda\nu) - d(\mu\nu) \in \Pp_\Lambda$. Since $v\Lambda w$ is nonempty for
all $v,w\in\Lambda^0$, it follows that there is a well-defined function $C : \Lambda^0
\times \Lambda^0 \to \Z^k/\Pp_\Lambda$ satisfying $C(r(\lambda), s(\lambda)) = d(\lambda)
+ \Pp_\Lambda$ for all $\lambda \in \Lambda$.

Choose $\mu \in u \Lambda v$ and $\nu \in v\Lambda w$, and note that
$$
C(u,w) = d(\mu\nu) + \Pp_\Lambda = d(\mu) + d(\nu) + \Pp_\Lambda = C(u,v) + C(v,w).
$$
Now $C(u,u) = C(u,u) + C(u,u)$ forces $C(u,u) = 0$, and then $0 = C(u,u) = C(u,v) +
C(v,u)$ forces $C(u,v) = -C(v,u)$ for all $u$ and $v$. The last statement follows.
\end{proof}

\begin{corollary}\label{cor:C<->gpd}
Let $\Lambda$ be a strongly connected finite $k$-graph, and $\G$ be the associated
groupoid. For $\lambda \in \Lambda$ and $w \in \Lambda^0$, we have $C(r(\lambda), w) =
d(\lambda) + C(s(\lambda), w)$, and for $(x,g,y) \in \G$ we have $C(r(x), r(y)) = g +
\Pp_\Lambda$.
\end{corollary}
\bp%
The first assertion follows from
$C(r(\lambda),w)-C(s(\lambda),w)=C(r(\lambda),s(\lambda))$ and the definition of~$C$. For
the second, take $(x,g,y) \in \G$, and pick $m,n \in \N^k$ with $m-n = g$ and
$\sigma^m(x) = \sigma^n(y)$. Put $\mu = x(0,m)$ and $\nu = x(0,n)$. Then $s(\mu) =
r(\sigma^m(x)) = r(\sigma^n(x)) = s(\nu)$. Hence $C(r(x), r(y)) = C(r(\mu), s(\mu)) -
C(r(\nu), s(\nu)) = m + \Pp_\Lambda - n + \Pp_\Lambda = g + \Pp_\Lambda$.
\ep%

The following is the main technical result we will need later.

\begin{proposition} \label{properiod}
Let $\Lambda$ be a strongly connected finite $k$-graph. There exists $p\in \Pp_\Lambda
\cap (\N\setminus\{0\})^k$ such that for all $v,w \in \Lambda^0$, we have $v\sim w$ if
and only if $v \Lambda^p w \not= \emptyset$.
\end{proposition}
\begin{proof}
Fix a vertex $u_0 \in \Lambda^0$. For each $v \in \Lambda^0$ fix paths $\lambda_v \in
v\Lambda u_0$ and $\mu_v \in u_0 \Lambda v$. For each $v,w \in \Lambda^0$, define
$g_{v,w} := d(\lambda_v\mu_w) \in \N^k$.

We define $p \in (\N \setminus\{0\})^k$ as follows. For each $v,w \in \Lambda$ with $v
\sim w$, we have $d(\lambda_v\mu_w) + \Pp_\Lambda = C(v,w) = 0 + \Pp_\Lambda$, and so
$g_{v,w} \in \Pp_\Lambda$. By Lemma~\ref{lem:Pp} there are cycles $\alpha,\beta$ in
$u_0\Lambda u_0$ with $d(\alpha) - d(\beta) = g_{v,w}$. Since $\Lambda$ has no sources,
there exists $\tau \in u_0\Lambda^{(1, \dots, 1)}$ and since $\Lambda$ is strongly
connected, there exists $\tau' \in s(\tau)\Lambda u_0$. Now $\alpha\tau\tau'$ and
$\beta\tau\tau'$ are cycles, and $m_{v,w} := d(\alpha\tau\tau')$ and $n_{v,w} :=
d(\beta\tau\tau')$ belong to $\Pp_{u_0}^+$. In particular, $m_{v,w}, n_{v,w} \in
\Pp_\Lambda \cap (\N \setminus \{0\})^k$ and $g_{v,w} = m_{v,w} - n_{v,w}$. Let
\[
p := \sum_{v\sim w} m_{v,w}.
\]
Since $\Pp_{u_0}^+$ is a semigroup, we have $p \in \Pp_\Lambda \cap (\N \setminus
\{0\})^k$. We show that $v\sim w$ if and only if $v\Lambda^p w \not= \emptyset$.

First suppose that $v \sim w$. We have
\[
p = \sum_{v' \sim w'} m_{v',w'}
    = g_{v,w} + \Big(n_{v,w} + \sum_{v' \sim w', (v',w')\not= (v,w)} m_{v', w'}\Big).
\]
Let $n := \Big(n_{v,w} + \sum_{v' \sim w', (v',w')\not= (v,w)} m_{v', w'}\Big)$. Then $n
\in \Pp_{u_0}^+$; say $\nu \in u_0 \Lambda^n u_0$. So $d(\lambda_v \nu \mu_w) = g_{v,w} +
n = p$, giving $\lambda_v\nu\mu_w \in v\Lambda^p w$.

Now suppose that $v\Lambda^p w \not= \emptyset$, say  $\lambda \in v\Lambda^p w$. Then
$C(v,w) = d(\lambda) + \Pp_\Lambda = p +\Pp_\Lambda = 0 + \Pp_\Lambda$, giving $v\sim w$.
\end{proof}

\begin{proposition} \label{pbij}
Let $\Lambda$ be a strongly connected finite $k$-graph. For $v \in \Lambda^0$, the map
$C(\cdot, v) : \Lambda^0 \to \Z^k/\Pp_\Lambda$ induces a bijection $\widetilde{C}_v :
\lms \to \Z^k/\Pp_\Lambda$.  There is a free and transitive action of $\Z^k/\Pp_\Lambda$
on~$\lms$ such that $\widetilde{C}_v\big((g + \Pp_\Lambda)\cdot [w]\big) = g +
\widetilde{C}_v([w])$ for all $g \in \Z^k$, $v \in \Lambda^0$ and $[w] \in \lms$.
\end{proposition}
\bp%
If $u \sim w$ then $C(u,w) = 0$, and so Lemma~\ref{twopaths} gives $C(u,v) = C(u,w) +
C(w,v) = C(w,v)$. So $C(\cdot,v)$ descends to a function $\widetilde{C}_v : \lms \to
\Z^k/\Pp_\Lambda$. If $\widetilde{C}_v([u]) = \widetilde{C}_v([w])$, then
Lemma~\ref{twopaths} gives $C(u,w) = C(u,v) + C(v,w) = C(u,v) - C(w,v) =
\widetilde{C}_v([u]) - \widetilde{C}_v([w]) = 0$ and so $u \sim w$. So $\widetilde{C}_v$
is injective. For surjectivity, fix $g+\Pp_\Lambda \in \Z^k/\Pp_\Lambda$, and express $g
= m-n$ with $m,n \in \N^k$. By~\eqref{eq:nosinks}, there exist $\lambda \in \Lambda^m v$,
and then $\mu \in r(\lambda)\Lambda^n$. Another application of Lemma~\ref{twopaths} gives
\[
\widetilde{C}_v([s(\mu)]) = C(s(\mu),v) = C(r(\lambda),v) - C(r(\lambda),s(\mu))
    = m-n+\Pp_\Lambda = g+\Pp_\Lambda,
\]
so $\widetilde{C}_v$ is surjective.

Pulling back the action of $\Z^k/\Pp_\Lambda$ on itself by translation along the
bijection $\widetilde{C}_v$ gives the desired free and transitive action on $\lms$.
Choose another $v' \in \Lambda^0$ and let $\diamond$ denote the action of
$\Z^k/\Pp_\Lambda$ such that $\widetilde{C}_{v'}\big((g + \Pp_\Lambda)\diamond[w]\big) =
g + \widetilde{C}_{v'}([w])$. Choose $u \in (g + \Pp_\Lambda)\cdot[w]$ and $u' \in (g +
\Pp_\Lambda)\diamond[w]$. Using Lemma~\ref{twopaths} repeatedly, we check that
\begin{align*}
C(u, u')
    &= C(u,v) + C(v, v') - C(u',v')
    = C(v,v') + \widetilde{C}_v\big((g + \Pp_\Lambda)\cdot [w]\big)
            - \widetilde{C}_{v'}\big((g + \Pp_\Lambda)\diamond[w]\big) \\
    &= C(v,v') + (g + \widetilde{C}_v([w])) - (g + \widetilde{C}_{v'}([w]))
    = C(v,v') + C(w,v) - C(w, v') = 0.
\end{align*}
So $(g + \Pp_\Lambda) \cdot [w] = (g + \Pp_\Lambda) \diamond [w]$ for all $w$.%
\ep

\begin{corollary}\label{cor:PersubsetP}
Let $\Lambda$ be a strongly connected finite $k$-graph. Then $\Pp_\Lambda$ has index at
most $|\Lambda^0|$ in~$\Z^k$, and $\Per\Lambda \subseteq \Pp_\Lambda$. We have
$\Pp_\Lambda=\Z^k$ if and only if $\Lambda$ is primitive.
\end{corollary}
\bp%
The bijection $\Z^k/\Pp_\Lambda \to \lms$ of Proposition~\ref{pbij} shows that
$|\Z^k/\Pp_\Lambda| \le |\Lambda^0|$. We have $\Per\Lambda \subseteq \{d(\mu) - d(\nu)
\mid r(\mu) = r(\nu)\text{ and }s(\mu) = s(\nu)\}$ by~\eqref{eper} and so
Lemma~\ref{twopaths} gives $\Per\Lambda \subseteq \Pp_\Lambda$.

Suppose that $\Lambda$ is primitive; say $v\Lambda^p w \not= \emptyset$ for all $v,w$.
Then each $v\Lambda^p v \not= \emptyset$, so $p \in \Pp_\Lambda$. For $v,w \in \Lambda^0$
and $\lambda \in v\Lambda^p w$, we have $C(v,w) = d(\lambda) + \Pp_\Lambda = p +
\Pp_\Lambda = 0 + \Pp_\Lambda$. So $\lms$, and hence also $\Z^k/\Pp_\Lambda$, is a
singleton, giving $\Pp_\Lambda = \Z^k$. Now suppose that $\Pp_\Lambda=\Z^k$. Then
$C(v,w) = 0$ for all $v,w$, and then $\Lambda$ is primitive by Proposition~\ref{properiod}.%
\ep

With $\sim$ as in Lemma~\ref{twopaths}, for an equivalence class $\omega\in\lms$ we
identify $\R^\omega$ with $\operatorname{span}\{\delta_v \mid v \in \omega\} \subseteq
\R^{\Lambda^0}$. So $\R^{\Lambda^0}$ decomposes as the internal direct sum
\begin{equation}\label{edirsum}\textstyle
\R^{\Lambda^0}=\bigoplus_{\omega\in\lms}\R^\omega.
\end{equation}

Observe that the action of Proposition~\ref{pbij} determines a transitive action of
$\N^k$ on $\lms$.

\begin{proposition} \label{pFrobenius}
Let $\Lambda$ be a strongly connected finite $k$-graph. The decomposition~\eqref{edirsum}
defines a system of imprimitivity for the semigroup of matrices~$(A^n)_{n\in\N^k}$; that
is, $A^n\R^\omega\subseteq \R^{n \cdot \omega} $ for all $n\in\N^k$ and $\omega\in\lms$.
\end{proposition}
\bp%
Take $\omega \in \lms$ and $v\in\omega$. If $|w\Lambda^nv| = A^n(w,v) > 0$, then
$C(w,v)=n+\Pp_\Lambda$ and hence the class of $w$ in $\lms$ is $n\cdot\omega$.%
\ep

For the next result, let $\Delta_k$ denote the $k$-graph with objects $\Z^k$, morphisms
$\{(g,h) \in \Z^k \times \Z^k \mid g \le h\}$, structure maps $r(g,h) = g$, $s(g,h) = h$
and $d(g,h) = h-g$, and composition $(g,h)(h,l) = (g,l)$. There is a free and transitive
action of $\Z^k$ on $\Delta_k$ by translation, and for any subgroup $G \le \Z^k$, the
quotient $\Delta_k/G$ is a $k$-graph under the inherited operations. If $G$ has finite
index, then $\Delta_k/G$ is finite and strongly connected and we think of it as a simple
$k$-dimensional cycle. We write $[g,h]$ for the image of $(g,h) \in \Delta_k$ in
$\Delta_k/G$, and we identify $(\Delta_k/G)^0$ with $\Z^k/G$ via $[g,g] \mapsto [g]$.

\begin{proposition}\label{prp:P=Per}
Let $\Lambda$ be a strongly connected finite $k$-graph. The following are equivalent.
\begin{enumerate}
\item\label{it:L=D/P} $\Lambda \cong \Delta_k/\Pp_\Lambda$;
\item\label{it:1in} $|v\Lambda^{e_i}| = 1$ for all $v\in \Lambda^0$ and $1\le i \le
    k$;
\item\label{it:P=Per} $\Pp_\Lambda = \Per\Lambda$; and
\item\label{it:rho=1} $\rho(\Lambda) = (1, \dots, 1)$.
\end{enumerate}
\end{proposition}
\bp%
\mbox{(\ref{it:L=D/P})$\implies$(\ref{it:1in}).} Take an isomorphism $\phi : \Lambda \to
\Delta_k/\Pp_\Lambda$. For $v \in \Lambda^0$, pick $g \in \Z^k$ with $\phi(v) = [g]$.
Then $|v\Lambda^{e_i}| = |[g](\Delta_k/\Pp_\Lambda)^{e_i}| = |\{[g, g+e_i]\}| = 1$.

\mbox{(\ref{it:1in})$\implies$(\ref{it:P=Per}).} We have $\Per\Lambda \subseteq
\Pp_\Lambda$ by Corollary~\ref{cor:PersubsetP}. Condition~(\ref{it:1in}) implies that
each $|v\Lambda^\infty| = 1$. So if $\mu$ is a cycle and $x \in s(\mu)\Lambda^\infty$,
then $\mu x \in s(\mu)\Lambda^\infty$ forces $\mu x = x$. Hence $\sigma^{d(\mu)}(x) = x$
for all $x \in s(\mu)\Lambda^\infty$, and \cite[Lemma~5.1]{AHLRS2} gives
$\sigma^{d(\mu)}(y) = y$ for all $y$. Since the degrees of cycles generate~$\Pp_\Lambda$
we obtain $\Pp_\Lambda \subseteq \Per\Lambda$.

\mbox{(\ref{it:P=Per})$\implies$(\ref{it:rho=1}).} We prove the contrapositive. Suppose
that $\rho(\Lambda) \not= (1,\dots,1)$. By~\eqref{eper}, $\Per\Lambda$ is contained in
the group of elements $g\in\Z^k$ such that $g\cdot\log\rho(\Lambda)=0$. Hence
$\Per\Lambda$ has infinite index in $\Z^k$, while $\Pp_\Lambda$ has finite index by
Corollary~\ref{cor:PersubsetP}.

\mbox{(\ref{it:rho=1})$\implies$(\ref{it:L=D/P}).} Proposition~\ref{properiod} gives $p
\in \N^k$ such that $v \sim w$ if and only if $v\Lambda^p w \not= \emptyset$. So $A^p$ is
block-diagonal with strictly positive diagonal blocks indexed by $\omega \in \lms$. Since
$\rho(A^p) = \rho(\Lambda)^p = 1$, each diagonal block of $A^p$ has spectral radius 1.
The only strictly positive square integer matrix with spectral radius~1 is the $1 \times
1$ identity matrix, so $A^p$ is the identity matrix, and $\sim$ is the trivial relation.
Now Propositions~\ref{pbij} and~\ref{pFrobenius} give a bijection $\phi : \Lambda^0 \to
\Z^k/\Pp_\Lambda$ such that $u\Lambda^n w \not= \emptyset$ implies $\phi(u)=n+\phi(w)$.
This shows that every column and row of the matrix $A^n$ has at most one nonzero entry.
Since $A^n$ has no zero columns and rows, and $\rho(A^n)=1$, this implies that $A^n$ is a
permutation matrix. Specifically, we have $|u\Lambda^n w| = \delta_{n + \Pp_\Lambda,
\phi(u) - \phi(w)}$. So $\lambda \mapsto [\phi(r(\lambda)), \phi(s(\lambda))]$ is a
bijection of $\Lambda$ onto $\Delta^k/\Pp_\Lambda$ that preserves range, source and
degree. Since each $|v\Lambda^n| = 1$ this bijection automatically preserves composition,
and hence is an isomorphism of
$k$-graphs.%
\ep

\section{Type classification of the \texorpdfstring{KMS$_1$}{KMS1}
    states of \texorpdfstring{$C^*(\Lambda)$}{C*(Lambda)}}
\label{sec:typeclass}

Let $\Lambda$ be a strongly connected finite $k$-graph, and let $C^*(\Lambda)^\gamma$ be
the fixed-point algebra for the gauge action $\gamma$ of $\T^k$ on $C^*(\Lambda)$. By
Lemma~\ref{lem:Rgamma}, we have $C^*(\Lambda)^\gamma \cong C^*(\RR^\gamma)$, and so
$C^*(\RR^\gamma)$ is AF (see \cite{KP}). Specifically, for each $n \in \N^k$ and $v \in
\Lambda^0$, put $\Ff_\Lambda(n,v) := \clsp\{s_\mu s^*_\nu \mid \mu,\nu \in \Lambda^n
v\}$, and then put $\Ff_\Lambda(n) := \clsp\{s_\mu s^*_\nu \mid \mu,\nu \in \Lambda^n\}$
for each $n$. The $s_\mu s^*_\nu$ in any given $\Ff_\Lambda(n,v)$ are matrix units, so
$\Ff_\Lambda(n,v) \cong \Mat_{\Lambda^n v}(\C)$, and we have $\Ff_\Lambda(n) =
\bigoplus_v \Ff_\Lambda(n,v)$. Relation~(CK) shows that if $\mu,\nu \in \Lambda^n v$ then
$s_\mu s^*_\nu = \sum_{\lambda \in v\Lambda^m} s_{\mu\lambda}s^*_{\nu\lambda} \in
\Ff_\Lambda(m+n)$ for all $m$. So each $\Ff_\lambda(n)\subset\Ff_\Lambda(m+n)$, with
inclusion matrix $A^m$, and the inductive limit of the algebras $\Ff_\Lambda(n)$ is
$C^*(\Lambda)^\gamma$.

Given a matrix $B\in\Mat_N(\N)$, let $\Ff_B$ denote the unital AF algebra whose Bratteli
diagram is stationary with $N$ vertices $\{v_{n,i} \mid 1 \le i \le N\}$ at level $n$,
and $B(i,j)$ edges connecting $v_{n,i}$ to $v_{n+1,j}$ for all $i,j$. That is, $\Ff_B =
\varinjlim C_n$, where $C_n = \bigoplus^N_{i=1} \Mat_{\sum_k B^n(k,i)}(\C)$, and the
partial inclusions $C_{n,i} \hookrightarrow C_{n+1, j}$ have multiplicity $B(i,j)$.

\begin{proposition}\label{prp:core decomp}
Let $\Lambda$ be a strongly connected finite $k$-graph, and let $\sim$ be the equivalence
relation on $\Lambda^0$ described in Lemma~\ref{twopaths}. Take $p \in (\N \setminus
\{0\})^k$ as in Proposition~\ref{properiod}, so $v \sim w$ if and only if $v\Lambda^p w
\not= \emptyset$. For $\omega \in \lms$, define $A^p_\omega \in \Mat_\omega(\N)$ by
$A^p_\omega(v,w) = |v\Lambda^p w|$ for $v,w\in\omega$, and define $q_\omega := \sum_{v
\in \omega} s_v \in C^*(\Lambda)$. Then each $A^p_\omega$ is primitive, the projections
$q_\omega$ are central in $C^*(\Lambda)^\gamma$, each $q_\omega C^*(\Lambda)^\gamma \cong
\Ff_{A^p_\omega}$, and $C^*(\Lambda)^\gamma = \bigoplus_{\omega \in \lms} q_\omega
C^*(\Lambda)^\gamma$.
\end{proposition}
\bp%
Each $A^p_\omega$ is primitive --- indeed its entries are all strictly positive --- by
choice of $p$. Since $p_i \not= 0$ for all $i$, the sequence $(np)_{n \in \N}$ is cofinal
in $\N^k$. Hence $C^*(\Lambda)^\gamma = \varinjlim_{n \in \N^k} \Ff_\Lambda(n) =
\varinjlim_{n\in\N} \Ff_\Lambda(np)$. As explained above, the inclusion $\Ff_\Lambda(np)
\hookrightarrow  \Ff_\Lambda((n+1) p)$ has matrix $A^p$, which is block-diagonal with
blocks $A^p_\omega$ by choice of $p$. So $C^*(\Lambda)^\gamma \cong \bigoplus_\omega
\Ff_{A^p_\omega}$ as claimed. Each $q_\omega$
is the identity projection in $A^p_\omega$, so is central.%
\ep

Since the $A^p_\omega$ are primitive, the stationary Bratteli diagrams they determine are
cofinal, so each $q_\omega C^*(\Lambda)^\gamma \cong \Ff_{A^p_\omega}$ is simple by
\cite[Corollary~3.5]{Bratteli}. It follows also that each $q_\omega C^*(\Lambda)^\gamma$
has a unique tracial state, see for example \cite[Proposition~10.4.9]{NS}. The trace
vector of the approximating subalgebra $\bigoplus_{v \in \omega} \Ff_\Lambda(np,v)$ of
$q_\omega C^*(\Lambda)^\gamma$ is given by $\rho(A^p_\omega)^{-n}\xi^\omega$, where
$\xi^\omega$ is the Perron-Frobenius eigenvector of~$A^p_\omega$ with unit 1-norm. As a
special case, we deduce that the following are equivalent:
\begin{enumerate}
\item $\Lambda$ is primitive;
\item $C^*(\Lambda)^\gamma$ is simple; and
\item $C^*(\Lambda)^\gamma$ carries a unique tracial state.
\end{enumerate}
Equivalence of the first two assertions has been already proved
in~\cite[Theorem~7.2]{MP}.

\begin{corollary} \label{cergodic}
Let $\Lambda$ be a strongly connected finite $k$-graph, and let $\sim$ be the equivalence
relation on $\Lambda^0$ described in Lemma~\ref{twopaths}. Let $\RR^\gamma$ be the
equivalence relation of Lemma~\ref{lem:Rgamma}. Then the ergodic components of
$\RR^\gamma$ with respect to $\mueq$ are the sets
\begin{equation} \label{exomega}
X_\omega := \{x\in \Lambda^\infty \mid r(x) \in \omega\}\ \ \text{for}\ \ \omega \in \lms.
\end{equation}
\end{corollary}
\bp%
The isomorphism $C^*(\Lambda)^\gamma \cong C^*(\RR^\gamma)$ carries each projection
$q_\omega$ described in Proposition~\ref{prp:core decomp} to the characteristic function
$\unb_{X_\omega}$ of $X_\omega$ regarded as a subset of the diagonal in $\RR^\gamma$.
Hence the projections $\unb_{X_\omega}$ are central in $C^*(\RR^\gamma)$, so the sets
$X_\omega$ are $\RR^\gamma$-invariant and $\Lambda^\infty = \bigsqcup_\omega X_\omega$.
For $\omega \in \lms$, let $\RR^\gamma_\omega := \RR^\gamma \cap (X_\omega \times
X_\omega)$, regarded as an equivalence relation on $(X_\omega, \mueq|_{X_\omega})$. Then
$\RR^\gamma_\omega$ is clopen in $\RR^\gamma$, and the canonical inclusion
$C_c(\RR^\gamma_\omega) \subseteq C_c(\RR^\gamma)$ extends to an isomorphism
$W^*(\RR^\gamma_\omega) \cong \unb_{X_\omega} W^*(\RR^\gamma)$. The formula $\varphi(f) =
\frac{1}{\mueq(X_\omega)} \int_{\Lambda^\infty} f(x,x)d\mueq(x)$ for $f$ supported on
$\RR^\gamma_\omega$ gives a normal tracial state on $W^*(\RR^\gamma_\omega)$. Uniqueness
of the trace on the dense subalgebra $C^*(\RR^\gamma_\omega) \cong q_\omega
C^*(\Lambda)^\gamma$ implies that $\varphi$ is a unique normal tracial state. Thus
$W^*(\RR^\gamma_\omega)$ is a factor, and so $\RR^\gamma_\omega$ is ergodic by
\cite[Proposition~2.9(2)]{FM2}. \ep

We can now compute the factor types of the extremal KMS states of $C^*(\Lambda)$.

\begin{theorem} \label{tmain}
Let $\Lambda$ be a strongly connected finite $k$-graph,  let $\alpha$ denote the
preferred dynamics on~$C^*(\Lambda)$, and let $\varphi_\chi$ be the extremal
$\alpha$-KMS$_1$-state on $C^*(\Lambda)$ corresponding to a character $\chi$ of
$\Per\Lambda$ as in \cite[Theorem 7.1]{AHLRS2}. Let $\Pp_\Lambda \subseteq\Z^k$ be the
group of periods of $\Lambda$ from Lemma~\ref{lem:Pp}. Then the Connes invariant $S$ of
$\pi_{\varphi_\chi}(C^*(\Lambda))''$ is
\[
S = \overline{\{\rho(\Lambda)^g \mid g \in \Pp_\Lambda\}} \subseteq [0,\infty).
\]
\end{theorem}
\bp%
By Proposition~\ref{pmainp} it suffices to show that $S(W^*(\RR)) =
\overline{\{\rho(\Lambda)^g \mid g \in \Pp_\Lambda\}}$. Let $\sim$ be as in
Lemma~\ref{twopaths}, and fix $\omega\in\lms$. Let $X_\omega$ be as in~\eqref{exomega}.
Since $W^*(\RR)$ is a factor, $S(W^*(\RR)) = S(\unb_{X_\omega}W^*(\RR)\unb_{X_\omega})$.
The corner $\unb_{X_\omega}W^*(\RR)\unb_{X_\omega}$ is also a factor, and is the von
Neumann algebra of the relativised equivalence relation $\RR_\omega :=
\RR\cap(X_\omega\times X_\omega)$ on $(X_\omega,\mueq|_{X_\omega})$.

Consider the relation $\RR^D_\omega \subseteq \RR_\omega$ as in~\eqref{eq:RD}, so
$$
x\sim_{\RR^D_\omega} y\text{ if and only if } x\sim_{\RR_{\omega}}y\text{ and } D(x,y)=1.
$$
Corollary~\ref{cergodic} implies that $\RR^\gamma_\omega = \RR^\gamma \cap
(X_\omega\times X_\omega) \subseteq \RR^D_\omega$ is ergodic, and so $\RR^D_\omega$ is
ergodic too. So Corollary~\ref{ceqConnes} applied to $\QQ = \RR_\omega$ implies that
$S(W^*(\RR)) = S(W^*(\RR_\omega))$ is the set of essential values of $D|_{\RR_\omega}$
with respect to the left counting measure $\nueq$ induced by $\mueq$. We must show that
this set is precisely $\overline{\{\rho(\Lambda)^g \mid g \in \Pp_\Lambda\}}$.

In order to prove that $\overline{\{\rho(\Lambda)^g \mid g \in \Pp_\Lambda\}}$ contains
the set of essential values, it suffices to show that $D(x,y)\in{\{\rho(\Lambda)^g \mid g
\in \Pp_\Lambda\}}$ for $\nueq$-a.e.~$(x,y)\in \RR_\omega$. Take $(x,y)\in \RR_\omega$
and choose $m,n \in \N^k$ with $(x, m-n, y) \in \G$. Since $r(x) \sim r(y)$,
Corollary~\ref{cor:C<->gpd} gives $m-n \in \Pp_\Lambda$. Hence, by~\eqref{RadNik}, we
$\nueq$-almost surely have
$$
D(x, y) = \rho(\Lambda)^{n-m} \in \{\rho(\Lambda)^g \mid g \in \Pp_\Lambda\}.
$$

Now fix $s \in \{\rho(\Lambda)^g \mid g \in \Pp_\Lambda\}$; say $m,n\in\N^k$ satisfy
$m-n\in \Pp_\Lambda$ and $s = \rho(\Lambda)^{n-m}$. For every $x\in X_\omega$ we can find
$y\in\Lambda^\infty$ such that $\sigma^m(x)=\sigma^n(y)$, so $(x,m-n,y)\in\G$. Then  by
Corollary~\ref{cor:C<->gpd} we have $r(x)\sim r(y)$, hence $y\in X_\omega$. In
particular, the projection of the closed set
\[
Z:=\{(x,y)\in X_\omega\times X_\omega\mid \sigma^m(x)=\sigma^n(y)\}
\]
onto the first coordinate is the entire set $X_\omega$. It follows that $\nueq(Z)>0$, and
since $D(x, y) = \rho(\Lambda)^{n-m}=s$ for $\nueq$-a.e.~$(x,y)\in Z$, we see that $s$ is
an essential value of $D$ on $\RR_\omega$. Since the set of essential values of $D$ is
closed, the result follows.
\ep

\begin{remark}
We computed the type of $\varphi_\chi$ without describing the centre of the centraliser
of $\varphi_\chi$ in $\pi_{\varphi_\chi}(C^*(\Lambda))''$. But such a description falls
easily out of our arguments: We want to understand the ergodic components of the
equivalence relation $\RR^D$ on $(\Lambda^\infty,\mueq)$ defined as in~\eqref{eq:RD}.
Since the ergodic components of $\RR^\gamma \subseteq \RR^D$ are the sets~$X_\omega$, the
ergodic components of $\RR^D$ are unions of these sets. In other words, the ergodic
components are defined by a coarser equivalence relation $\approx$ than the relation
$\sim$ on $\Lambda^0$. A moment's reflection shows that this equivalence relation must be
given by $v\approx w$ if and only if there are $\lambda,\mu \in \Lambda$ with
$$
v=r(\lambda),\quad w=r(\mu),\quad s(\lambda)=s(\mu)\quad\text{and}\quad
    \rho(\Lambda)^{d(\lambda)-d(\mu)}=1.
$$
So the minimal central projections of the centraliser are the images of the projections
$\sum_{v\in\omega}s_v$ for $\omega\in\lmss$ (cf.~\cite{O}).
\end{remark}

\begin{proof}[Proof of Theorem~\ref{thm:main}]
Part~(\ref{it:b>1}) follows immediately from Proposition~\ref{prp:toeplitzfactors}, so
suppose that $\beta = 1$. Then \cite[Corollary~4.6]{AHLRS2} shows that $\phi$ factors
through a state $\overline{\phi}$ of $C^*(\Lambda)$, and \cite[Theorem~7.1]{AHLRS2}
implies that $\overline{\phi} = \varphi_\chi$ for some character $\chi$ of
$\Per(\Lambda)$.

If $\rho(\Lambda) = (1, \dots, 1)$, then Proposition~\ref{prp:P=Per} shows that $\Lambda
\cong \Delta_k/\Pp_\Lambda$. So each $v\Lambda^n$ and hence each $v\Lambda^\infty$ is a
singleton set. Choose $u_0 \in \Lambda^0$ and fix paths $\lambda_v \in v\Lambda u_0$
indexed by $v \in \Lambda^0$. Let $x$ be the unique infinite path with range $u_0$. Then
$(\lambda_v x, d(\lambda_v) - d(\lambda_w), \lambda_w x) \in \G$ for all $v,w$, and it
follows that the equivalence relation $\RR$ is the full equivalence relation $\Lambda^0
\times \Lambda^0$. Hence $W^*(\RR) \cong \Mat_{\Lambda^0}(\C)$. This proves
statement~(\ref{it:b=1;r=1}).

Now suppose that $\rho(\Lambda) \not= (1, \dots, 1)$. Then Theorem~\ref{tmain} shows that
the Connes invariant of $\pi_{\overline{\phi}}(C^*(\Lambda))''$ is the closure of the set
$S$ described in statement~(\ref{it:b=1;r>1}). Since $\Pp_\Lambda$ is a finite index
subgroup of $\Z^k$, this set has nontrivial intersection with $(0,1)$, and the result
follows.
\end{proof}

\begin{remark}\label{rmk:Yang}
Theorem~\ref{thm:main} generalises Yang's result \cite[Theorem~5.3]{Y} to finite
$k$-graphs with more than one vertex, and removes the technical hypothesis that the
intrinsic group of the $k$-graph has rank at most one. To see this, observe that if
$\Lambda$ has just one vertex, then $\Pp_\Lambda = \Z^k$, and $\rho(\Lambda) =
(|\Lambda^{e_1}|, \dots, |\Lambda^{e_k}|)$. So Theorem~\ref{thm:main} shows that the
Connes invariant of the associated factor is the closure of the subgroup of $(0,
\infty)^\times$ generated by the numbers $|\Lambda^{e_1}|, \dots, |\Lambda^{e_k}|$. Yang
considers only the situation where each $|\Lambda^{e_i}| \ge 2$. In this case, just as
Yang says, if some $\log(|\Lambda^{e_i}|)/\log(|\Lambda^{e_j}|)$ is irrational, the
factor is of type~$\mathrm{III}_1$. Otherwise we can uniquely write each
$|\Lambda^{e_1}|^{a_i} = |\Lambda^{e_i}|^{b_i}$ with $\gcd(a_i, b_i) = 1$, and then the
factor is of type~$\mathrm{III}_\lambda$, where $\lambda =
|\Lambda^{e_1}|^{-1/\operatorname{lcm}(b_2,\cdots, b_k)}$.
\end{remark}

\begin{remark}
As discussed in Remark~\ref{rmk:Pp dependence}, the group $\Pp_\Lambda$ depends only on
the skeleton of $\Lambda$ and is independent of the factorisation property. The same is
true of the spectral-radius vector $\rho(\Lambda)$. So Theorem~\ref{thm:main} shows that
the type of the factors obtained from extremal KMS states depends only on the skeleton of
$\Lambda$ and not on the factorisation property.
\end{remark}

We finish with an explicit example of the phenomenon described in the preceding remark: a
pair of $2$-graphs $\Lambda_1, \Lambda_2$ with the same adjacency matrices, and hence
determining the same factors, in which $\Per\Lambda_1$ and $\Per\Lambda_2$ are distinct
proper subgroups of the common group of periods, which is itself a strict subgroup of
$\Z^2$.

\begin{example}
Consider the $2$-coloured graph below.
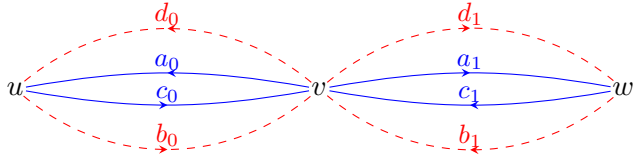
\[\begin{tikzpicture}[>=stealth, decoration={markings, mark=at position 0.5 with {\arrow{>}}}]
    \node[circle, inner sep=0pt] (u) at (0,0) {$u$};
    \node[circle, inner sep=0pt] (v) at (4,0) {$v$};
    \node[circle, inner sep=0pt] (w) at (8,0) {$w$};
    \draw[blue, postaction=decorate, out=170, in=10] (v) to node[pos=0.5, above, inner sep=1pt] {\small$a_0$} (u);
    \draw[blue, postaction=decorate, out=10, in=170] (v) to node[pos=0.5, above, inner sep=1pt] {\small$a_1$} (w);
    \draw[blue, postaction=decorate, out=350, in=190] (u) to node[pos=0.5, above, inner sep=1pt] {\small$c_0$} (v);
    \draw[blue, postaction=decorate, out=190, in=350] (w) to node[pos=0.5, above, inner sep=1pt] {\small$c_1$} (v);
    \draw[red, dashed, postaction=decorate, out=140, in=40] (v) to node[pos=0.5, above, inner sep=1pt] {\small$d_0$} (u);
    \draw[red, dashed, postaction=decorate, out=40, in=140] (v) to node[pos=0.5, above, inner sep=1pt] {\small$d_1$} (w);
    \draw[red, dashed, postaction=decorate, out=320, in=220] (u) to node[pos=0.5, above, inner sep=1pt] {\small$b_0$} (v);
    \draw[red, dashed, postaction=decorate, out=220, in=320] (w) to node[pos=0.5, above, inner sep=1pt] {\small$b_1$} (v);
\end{tikzpicture}\]
Any $2$-graph with this skeleton must satisfy $a_i b_i = d_i c_i$ for $i = 0,1$ and $a_i
b_{1-i} = d_i c_{1-i}$, and so there are two possible $2$-graphs with this skeleton: the
$2$-graph $\Lambda_1$ in which $c_i d_i = b_i a_i$ for $i=0,1$; and the $2$-graph
$\Lambda_2$ in which $c_i d_i = b_{1-i} a_{1-i}$ for $i=0,1$ (see \cite{HRSW}). Every
cycle $\mu$ in either $\Lambda_1$ or $\Lambda_2$ satisfies $d(\mu)_1 + d(\mu)_2 \in 2\Z$.
Since $d(a_1c_1) = (2,0)$ and $d(a_1b_1) = (1,1)$ generate $\{m \mid m_1 + m_2 \in
2\Z\}$, we see that $\Pp_{\Lambda_i} = \{m \mid m_1 + m_2 \in 2\Z\}$ for $i = 1,2$.

It is easy to see that $\Lambda_1$ is the pullback of the $1$-graph $E$ consisting of
blue (solid) paths in $\Lambda$ over $f : \N^2 \to \N $, $(m,n) \mapsto m+n$, as in
\cite[Definition~1.9]{KP}. Since every cycle in $E$ has an entrance, the periodicity
group of $E$ is trivial, and one can use this to check that $\Per\Lambda_1 = \Z(-1,1)
\subsetneq \Pp_{\Lambda_1} \subsetneq \Z^2$.

We claim that $\Per\Lambda_2 = 2\Z(-1,1)$. To see this, first note that every infinite
path $x \in u\Lambda^\infty_2$ satisfies $\sigma^{(1,0)}(c_1d_1c_1x) = d_1c_1 x = a_1 b_1
x \in Z(a_1)$ whereas $\sigma^{(0,1)}(c_1d_1c_1 x) = \sigma^{(0,1)}(b_2a_2c_1 x) = a_2c_1
x \in Z(a_2)$. Since $Z(a_1)$ and $Z(a_2)$ are disjoint, we deduce that $(1,0) - (0,1)
\not\in \Per\Lambda_2$. Since $\Per\Lambda_2 \subseteq \Pp_{\Lambda_2}$, we have $(m,0)
\not\in \Per\Lambda_2$ for $m$ odd, and if $m$ is even and nonzero then any path of the
form $(c_0a_0)^{m/2} c_1 x$ satisfies $(c_0a_0)^{m/2} c_1 x \in Z(c_0)$ and
$\sigma^{(m,0)}((c_0a_0)^{m/2} c_1 x) \in Z(c_1)$ giving $(m,0) \not\in \Per\Lambda_2$.
The same argument gives $(0,m) \not\in \Per\Lambda_2$. So $\Per\Lambda_2 = \Z n$ for some
$n \in (\Z\setminus\{0\})^2$. We have $v\Lambda^\infty_2 = \{b_{i_0}a_{i_0}b_{i_1}a_{i_1}
\dots \mid (i_n)^\infty_{n=1} \in \{0,1\}^\N\}$, and repeated application of the
factorisation rules using this description shows that $c_ia_i x = b_{1-i}d_{1-i} x$ for
$i = 1,2$ and $x \in v\Lambda^\infty_2$. Since $Z(v) = Z(c_1 a_1) \cup Z(c_2a_2)$, it
follows that $\sigma^{(2,0)}(x) = \sigma^{(0,2)}(x)$ for all $x \in v\Lambda^\infty_2$,
and so $(2,-2) \in \Per\Lambda_2$ by \cite[Lemma~5.1]{AHLRS2}, and therefore
$\Per\Lambda_2 = 2\Z(1,-1)$ as claimed.

For either of $\Lambda_1$ or $\Lambda_2$, we have
\[
A_1 = A_2 = \left(\begin{matrix} 0 & 1 & 0 \\ 1 & 0 & 1 \\ 0 & 1 & 0\end{matrix}\right),
\]
So $A_1(1, \sqrt2, 1)^t = \sqrt2(1, \sqrt2, 1)^t$, and since $A_1$ is irreducible, the
Perron-Frobenius theorem gives $\rho(\Lambda_i) = (\sqrt2, \sqrt2)$ for $i=1,2$. By
inspection, for $n \in \N^2$, there is a cycle in $\Lambda_i$ of degree $n$ if and only
if $n_1 + n_2 \in 2\Z$, so that $\rho(\Lambda_i)^n = 2^{(n_1+n_2)/2}$ is a power of $2$.
So Theorem~\ref{thm:main}(\ref{it:b=1;r>1}) says that for $i=1,2$ and any extremal
$\alpha$-KMS$_1$ state $\phi$ of $C^*(\Lambda_i)$, the factor
$\pi_\phi(C^*(\Lambda_i))''$ is of type~III$_{1/2}$.
\end{example}

\begin{remark}
We could modify the $2$-coloured graph of the preceding example by replacing every red
(dashed) edge $f$ with a pair of parallel red edges $(f,0), (f,1)$. For either of $i
=0,1$, we could then specify factorisation rules on this amplified $2$-coloured graph by
$e(f,j) = (f',j)e'$ ($j = 0,1$) whenever $ef = f'e'$ in $\Lambda_i$ above, obtaining a
new $2$-graph $\Lambda_{i,2}$. We then have $\rho(\Lambda_{i,2}) = (\sqrt2, 2\sqrt2)$.
Again the cycles in $\Lambda_{i,2}$ all satisfy $d(\lambda)_1 + d(\lambda_2) \in 2\Z$, so
each $\rho(\Lambda_{i,2})^{d(\lambda)} = 2^{(d(\lambda_1)+d(\lambda_2))/2 +
d(\lambda)_2}$ is a power of $2$. So we deduce from
Theorem~\ref{thm:main}(\ref{it:b=1;r>1}) that each $\pi_\phi(C^*(\Lambda_{i,2}))''$ is
still of type~III$_{1/2}$.

If, instead of replacing each red edge with two edges $(f,0)$ and $(f,1)$ we insert three
parallel red edges $(f,0), (f,1), (f,2)$ to obtain $2$-graphs $\Lambda_{i,3}$, we obtain
$\rho(\Lambda_{i,3}) = (\sqrt2, 3\sqrt2)$. Since $\log2$ and $\log3$ are rationally
independent, so are $\log\sqrt2 = \frac12\log2$ and $\log(3\sqrt2) = \log3 +
\frac12\log2$. So Theorem~\ref{thm:main}(\ref{it:b=1;r>1}) says that
$\pi_\phi(C^*(\Lambda_{i,3}))''$ is the injective~III$_1$ factor for $i = 1,2$.
\end{remark}

\bigskip

\end{document}